\documentclass[12pt]{amsart}
\usepackage{amsmath, amssymb, geometry}
\usepackage[usenames,dvipsnames]{xcolor} 
\usepackage{array} 
\usepackage{colortbl} 
\usepackage[colorinlistoftodos]{todonotes}
\usepackage{hyperref}
\usepackage{graphicx}
\graphicspath{ {Figures/} }

\newcommand{\CC}{\mathbb{C}}
\newcommand{\QQ}{\mathbb{Q}}
\newcommand{\RR}{\mathbb{R}}
\newcommand{\ZZ}{\mathbb{Z}}
\newcommand{\mcP}{\mathcal{P}}
\newcommand{\mcN}{\mathcal{N}}
\newcommand{\eps}{\varepsilon}

\DeclareMathOperator{\Gal}{Gal}

\theoremstyle{plain}
\newtheorem{thm}{Theorem}

\newtheorem{lemma}[thm]{Lemma}
\newtheorem{cor}[thm]{Corollary}

\theoremstyle{definition}
\newtheorem{defn}[thm]{Definition}

\newtheorem{conjecture}[thm]{Conjecture}

\theoremstyle{remark}

\newtheorem{ex}[thm]{Example}

\numberwithin{equation}{section}
\numberwithin{thm}{section}

\title{Stability for Take-Away Games}
\author{Simon Rubinstein-Salzedo}
\address{(Rubinstein-Salzedo): Euler Circle, Palo Alto, CA 94306}
\email{simon@eulercircle.com}
\author{Sherry Sarkar}
\address{(Sarkar): School of Mathematics, Georgia Institute of Technology, Atlanta, GA 30332}
\email{sherrys@gatech.edu}
\date{\today}

\begin{document}
\maketitle 

\begin{abstract}
In this paper, we study a family of take-away games called $\alpha$-\textsc{tag}, parametrized by a real number $\alpha\ge 1$. We show that for any given $\alpha$, there is a half-open interval $I_\alpha$ containing $\alpha$ such that the set of losing positions for $\alpha$-\textsc{tag} is the same as the set of losing positions for $\beta$-\textsc{tag} if and only if $\beta\in I_\alpha$. We then end with some results and conjectures on the nature of these intervals.
\end{abstract}

\section{Introduction}

In this paper, we study the losing positions of a certain family of games, known as \emph{take-away games}. In our study, the games are indexed by a single parameter $\alpha$, which is a real number greater than or equal to 1. It is also possible to study more general families of take-away games, as has been done in~\cite{Zieve96}.

Here are the rules for the games we study. Let $\alpha\ge 1$ be a real number. We define \textsc{$\alpha$-tag} (short for \textsc{$\alpha$-take-away game}) to be the two-player game played with following rules: 
\begin{enumerate}
\item The game begins with $n$ stones in one pile, for some nonnegative integer $n$. A \emph{move} in this game consists of removing at least one stone from the pile.
\item The two players alternate making moves.
\item The first player may take up to $n - 1$ stones.
\item After the first turn, a player can take up to $\alpha$ times the number of stones taken by the previous player on the last turn. 
\end{enumerate}

The winner of this game is the player who removes the last stone, or, more precisely, the loser is the player who is not able to remove a stone. (For instance, if $n=0$ or $n=1$, then the first player is not able to remove a stone, but the winner did not necessarily remove the last stone.)

Since the game is symmetric in the two players, there are only two possible outcomes for $\alpha$-\textsc{tag}, assuming optimal play: either the first player has a winning strategy, or the second player has a winning strategy. In accordance with standard combinatorial game theory parlance, we call a position in which the first player has a winning strategy an $\mcN$ position, and a position in which the second player has a winning strategy a $\mcP$ position.

There is a useful recursive way of determining which positions are $\mathcal{N}$ positions and which are $\mathcal{P}$ positions, thanks to the following lemma:

\begin{lemma}[{\cite[Theorem 2.13]{ANW07}}] A position is an $\mathcal{N}$ position iff there exists a move to a $\mathcal{P}$ position. A position is an $\mathcal{P}$ position iff all moves lead to $\mathcal{N}$ positions. \end{lemma}





Studying any impartial combinatorial game like $\alpha$-\textsc{tag} means determining which positions are the $\mcP$ positions and which are the $\mcN$ positions. Since in a typical game most positions are $\mcN$ positions, it is customary to focus on determining the smaller set of $\mcP$ positions. Formally, a position in \textsc{$\alpha$-tag} consists of two pieces of information: the \emph{pile size} (i.e.\ the number of stones remaining), and the \emph{move dynamic} (i.e.\ the maximum number of stones that may be removed on the next turn). However, in the current work, we are solely interested in determining the outcome class ($\mcN$ or $\mcP$) of the initial position, so we will be able to simplify our analysis by working only with the pile size, with a bit of care. 
 
\begin{defn} Let $T(\alpha)$ be the sequence of pile sizes $n$ such that the only move a player can make to win \textsc{$\alpha$-tag} in a pile of size $n$ with optimal play is to remove all remaining stones. \end{defn}

We note, of course, that during game play, it may not be possible to remove all the stones from a pile of size $n$; whether that move is allowable or not depends on the last move played. We also note that $T(\alpha)$ consists of exactly those $n$ such that the initial position of \textsc{$\alpha$-tag} with $n$ stones is a $\mcP$ position.

Schwenk in~\cite{Schwenk70} showed that the sequence $T(\alpha)$ can be enumerated by a sequence which eventually satisfies a simple recurrence of the form $P_n = P_{n-1} + P_{n-k}$ for some $k$, for sufficiently large values of $n$; see Theorem~\ref{thm:eventualrec}.

The main result in this paper is Theorem~\ref{thm:stability}, which says that the sequences $T(\alpha)$ change in discrete intervals based on $\alpha$. For instance, if $1\le\alpha<2$, then $T(\alpha)=(0,1,2,4,8,16,32,\ldots)$ consists of 0 together with the powers of 2. Similarly, when $2\le\alpha<\frac{5}{2}$, then $T(\alpha)=(0,1,2,3,5,8,13,21,\ldots)$ consists of the Fibonacci numbers. We think of this as a stability theorem for take-away games: even though the rules and allowable moves in the game differ whenever we change $\alpha$ even slightly (for sufficiently large $n$), these extra options do not change the optimal outcomes of the game. Most of the paper is devoted to proving this theorem, and then we end with some further results and questions about the nature of these stable intervals.

\section{History}

One commonly studied game, first introduced by Whinihan in~\cite{Whinihan63}, is the $\alpha = 2$ version of the game described above, or better known as \textsc{Fibonacci Nim}. The $T(\alpha)$ positions for this game are the Fibonacci numbers. \textsc{Fibonacci Nim} is interesting because its winning strategy relies on the following theorem:

\begin{thm}[Zeckendorf, \cite{Lek52,Zeckendorf72}] Every positive integer can be uniquely expressed as the sum of pairwise nonconsecutive Fibonacci numbers. \end{thm}

Zeckendorf's Theorem together with the following Lemma provides us with a winning strategy for \textsc{Fibonacci Nim}:

\begin{lemma} \label{lem:fibineq} For $i\ge 2$, we have $F_{i + 1} \leq 2 \cdot F_i < F_{i + 2}$. \end{lemma} 

One can construct a winning strategy for any positive non-Fibonacci integer by combining Zeckendorf's Theorem with Lemma~\ref{lem:fibineq}. Suppose that there are $n$ stones. We look at the Zeckendorf representation of $n$, say
\[ n = F_{i_1} + F_{i_2} + \cdots + F_{i_k}, \] where for each $j$ with $1\le j\le k-1$ we have $i_{j+1}-i_j\ge 2$. If $k\ge 2$, then a winning strategy for the first player is to remove the smallest part of the Zeckendorf representation, i.e.\ $F_{i_1}$. Due to Lemma~\ref{lem:fibineq}, the second player will not be able to remove the entire next Zeckendorf part. Since all Fibonacci numbers are $T(\alpha)$ positions, the second player is forced to play essentially in the next term $F_{i_2}$, and lose in that part. We will see this line of reasoning again when we study the $T(\alpha)$ positions of the general \textsc{$\alpha$-tag}.

\section{$\mathcal{P}$ Positions of \textsc{$\alpha$-tag}}

In the previous section, we computed the sequence $T(2)$ and showed that it is the sequence of Fibonacci numbers. Next, we consider the sequence $T(\alpha)$ for an arbitrary real number $\alpha\ge 1$. The computation of the sequence $T(\alpha)$ relies on a generalization of Zeckendorf's Theorem, first introduced by Schwenk in~\cite{Schwenk70}. Following~\cite{Schwenk70}, we generate a sequence $P^{\alpha}$ as follows. Let the first two terms of $P^{\alpha}$ be $P^{\alpha}_0 = 0$, $P^{\alpha}_1 = 1$. Then define
\[ P^{\alpha}_{k + 1} = P^{\alpha}_k + P^{\alpha}_j, \]
where $j$ is the the unique index such that
\[ \alpha \cdot P^{\alpha}_j \geq P^{\alpha}_k>\alpha\cdot P_{j-1}^\alpha. \]

There is a generalization of Zeckendorf's Theorem based on the sequence $P^\alpha$:

\begin{thm}[Generalized Zeckendorf's Theorem, \cite{Schwenk70}] Any positive integer $n$ can be uniquely expressed as a sum of terms of the sequence $P$ with the following condition
\[ n = P^{\alpha}_{i_1} + P^{\alpha}_{i_2} + \cdots + P^{\alpha}_{i_k} \quad \text{where} \quad \alpha \cdot P^{\alpha}_{i_j} < P^{\alpha}_{i_{j + 1}} \quad \text{for all } j<k. \]
\end{thm}
The proof is very similar to that of the classical Zeckendorf Theorem.  

\begin{thm}[\cite{Schwenk70}]
For any $\alpha\ge 1$, the sequence $T(\alpha)$ is equal to the sequence $(P^\alpha_i)$. 
\end{thm}

The details of the proof can be found in Schwenk's paper. From now on, we will refer to $P_i^\alpha$ instead of $T(\alpha)$ for this sequence. When $\alpha$ is fixed or clear from context, we shall simply write $P_i$ instead of $P_i^\alpha$.

\begin{defn} 
The \textit{window} $W_\alpha(P^\alpha_i)$ of a term $P^\alpha_i$ is \[W_\alpha(P^\alpha_i)=\{P^\alpha_j\in T(\alpha):\alpha\cdot P^\alpha_{i-1}<P^\alpha_j\le\alpha\cdot P^\alpha_i\}.\]
\end{defn}

For some $P=P^\alpha_i\in T(\alpha)$, the window $W_\alpha(P)$ is the set of $Q=Q^\alpha_j\in T(\alpha)$ such that $P+Q=Q^\alpha_{j+1}$ is the next term in $T(\alpha)$. For $P$ occurring early in the sequence $T(\alpha)$, $W_\alpha(P)$ may contain several elements. However, for sufficiently large values of $P\in T(\alpha)$, the $W_\alpha(P)$ consists of just a single element, and this is what causes the sequence of $T(\alpha)$ positions to satisfy a simple recurrence:

\begin{thm}[Schwenk] \label{thm:eventualrec} Fix $\alpha\ge 1$. Then there exists an integer $k$ such that, for sufficiently large values of $n$, we have $P^\alpha_n=P^\alpha_{n-1}+P^\alpha_{n-k}$. \end{thm}

\begin{cor} For $n$ sufficiently large, $W_\alpha(P^\alpha_n)$ is a set of size 1. \end{cor}

\section{Lemmas about Linear Recurrences}

In this section, we present some general lemmas about linear recurrences, as well as some about the specific family that are relevant to \textsc{$\alpha$-tag}; we provide references to the literature when we were able to find other sources for them.

\begin{lemma} \label{lem:linrecgenterm} Let $a_0,a_1,\ldots$ be a sequence of complex numbers satisfying a linear recurrence relation $a_{n+k}=c_{k-1}a_{n+k-1}+c_{k-2}a_{n+k-2}+\cdots+c_0a_n$ for all sufficiently large $n$. Let $\chi(x)=x^k-c_{k-1}x^{k-1}-c_{k-2}x^{k-2}-\cdots-c_0$ be the characteristic polynomial of the eventual recurrence, and let $r_1,\ldots,r_k$ be its complex roots, with multiplicity. If all the $r_i$'s are distinct, then there exist $\beta_1,\ldots,\beta_k\in\CC$ such that \[a_n=\beta_1r_1^n+\beta_2r_2^n+\cdots+\beta_kr_k^n\] for all sufficiently large $n$. \end{lemma}

See~\cite[Theorem 4.1.1]{Stanley12} for a proof.

From now on, we shall arrange the $r_i$'s in decreasing order of magnitude: $|r_1|\ge|r_2|\ge\cdots\ge|r_k|$.

\begin{lemma} \label{lem:posdom} With the notation of Lemma~\ref{lem:linrecgenterm}, suppose that all the $r_i$'s are distinct. Suppose furthermore that all the $\beta_i$'s are nonzero. If $a_n>0$ for all sufficiently large $n$, then $r_1$ is positive and real, $r_1>|r_2|$, and $\beta_1>0$. We call $r_1$ the \emph{positive dominant root}. \end{lemma}

See~\cite[Theorem 1]{BW81} for a proof.

\begin{lemma} \label{lem:galois} With the notation of Lemma~\ref{lem:linrecgenterm}, suppose that all the $r_i$'s are distinct. Suppose also that the $a_i$'s are all integers. Suppose that $\chi(x)$ factors over $\mathbb{Q}$ as $\chi(x)=\chi_1(x)\chi_2(x)\cdots\chi_j(x)$, where each $\chi_i(x)$ is irreducible over $\mathbb{Q}$. If $r_{i_1},\ldots,r_{i_d}$ are the roots of $\chi_1(x)$, then either $\beta_{i_1}=\beta_{i_2}=\cdots=\beta_{i_d}=0$, or else all of $\beta_{i_1},\ldots,\beta_{i_d}$ are nonzero. \end{lemma}

\begin{proof} By~\cite[Proposition 4.2.2]{Stanley12}, the generating function for $a_n$ has the form \[\sum_{n=0}^\infty a_nx^n=R(x)+\frac{\beta_{i_1}}{1-r_{i_1}x}+\cdots+\frac{\beta_{i_k}}{1-r_{i_k}x},\] where $R(x)\in\ZZ[x]$. Let $K$ be the Galois closure of $\mathbb{Q}(\beta_{i_1},\ldots,\beta_{i_k},r_{i_1},\ldots,r_{i_k})(x)$ over $\QQ(x)$, and let $\sigma\in\Gal(K/\mathbb{Q}(x))$ be an arbitrary element. Then $\sigma$ permutes $r_{i_1},\ldots,r_{i_d}$, and since $\sum_{n=0}^\infty a_nx^n$ is fixed by $\sigma$, we must have \[\sigma\left(\frac{\beta_{i_1}}{1-r_{i_1}x}\right)=\frac{\beta_{i_j}}{1-r_{i_j}x}\] for some $j$ with $1\le j\le d$. Furthermore, $\Gal(K/\QQ(x))$ acts transitively on the terms $\frac{\beta_{i_j}}{1-r_{i_j}x}$, so for each $j$ with $1\le j\le d$, there is some $\sigma\in\Gal(K/\QQ(x))$ that sends $\frac{\beta_{i_1}}{1-r_{i_1}x}$ to $\frac{\beta_{i_j}}{1-r_{i_j}x}$. Thus if $\beta_{i_1}\neq 0$, then $\beta_{i_k}\neq 0$ for $1\le j\le d$, and vice versa. \end{proof}

\begin{lemma} \label{lem:irreducible} For all $k\ge 2$, $k\not\equiv 5\pmod{6}$ the polynomial $x^k-x^{k-1}-1$ is irreducible over $\mathbb{Q}$. When $k\equiv 5\pmod{6}$, then $x^k-x^{k-1}-1$ factors as $x^2-x+1$ times an irreducible factor. \end{lemma}

\begin{proof} Selmer in~\cite{Selmer56} shows that the polynomial $f(x)=x^k-x-1$ is irreducible for all $k\ge 2$, and that $g(x)=x^k+x+1$ is irreducible when $k\not\equiv 2\pmod{3}$, and factors as $x^2+x+1$ times an irreducible factor when $k\equiv 2\pmod{3}$. When $k$ is even, we have $x^k-x^{k-1}-1=-x^kf(-\frac{1}{x})$, so it is irreducible. When $k$ is odd, we have $x^k-x^{k-1}-1=x^kg(-\frac{1}{x})$, so it is irreducible when $k\not\equiv 5\pmod{6}$ and factors as $x^2-x+1$ times an irreducible factor when $k\equiv 5\pmod{6}$. \end{proof}

\begin{lemma} \label{lem:2permag} If $k\ge 2$, then the polynomial $x^k-x^{k-1}-1$ contains at most two roots of any given magnitude. \end{lemma}

\begin{proof} Selmer shows in~\cite{Selmer56} that on any circle $|x|=r$ in the complex plane, the polynomials $x^k\pm(x+1)$ have only at most two roots. Since the roots of $x^k-x^{k-1}-1$ are the negative reciprocals of the roots of $x^k\pm(x+1)$ (depending on the parity of $k$), it follows that these polynomials also have at most two roots on any given circle $|x|=r$. \end{proof}

\begin{lemma} \label{lem:r2stuff} Let $k\ge 6$. With notation as in Lemma~\ref{lem:linrecgenterm}, if $a_n=a_{n-1}+a_{n-k}$ for all sufficiently large $n$, then $|r_2|>1$, and $r_2$ is nonreal. \end{lemma}

\begin{proof} First, note that $r_1>1$, because the product of the roots is equal to $\pm 1$, so some root (and in particular the largest in absolute value) must have absolute value at least 1. Now suppose for some $k\ge 6$, we have that $|r_2|\le 1$. We consider two cases: $|r_2|=1$ and $|r_2|<1$. Suppose first that $|r_2|<1$. Then $r_1$ is a Pisot number, i.e.\ a real algebraic integer greater than 1, all of whose Galois conjugates have absolute value less than 1. The smallest Pisot number is the positive root of $x^3-x-1$, or $1.3247\ldots$ (see~\cite{Siegel44}). However, for every $k\ge 6$, $1.3^k-1.3^{k-1}-1>0$ whereas $1^k-1^{k-1}-1=-1<0$, so $1<r_1<1.3$. Thus $r_1$ cannot be a Pisot number.

Suppose now that $|r_2|=1$. If $k\equiv 5\pmod{6}$, then Lemmas~\ref{lem:irreducible} and~\ref{lem:2permag} imply that $r_2$ and $r_3$ are the primitive sixth roots of unity, and that $|r_4|<1$. This means that $r_1$ is again a Pisot number. However, this cannot be the case for the same reason as before, as $r_1$ is smaller than the smallest Pisot number. On the other hand, if $k\not\equiv 5\pmod{6}$ and $|r_2|=1$, then $r_2$ is a Galois conjugate of $r_1$, so $r_1$ is a Salem number, i.e.\ an algebraic integer greater than 1 all of whose conjugates have absolute values at most 1, with at least one of the conjugates having an absolute value equal to 1. The minimal polynomial of any Salem number is a reciprocal polynomial, i.e.\ a polynomial $p(x)$ such that $p(x)=x^{\deg(p)}p(\frac{1}{x})$ (see~\cite[\S 6]{Salem45}). Since $x^k-x^{k-1}-1$ is not a reciprocal polynomial, $r_1$ cannot be a Salem number. Thus $|r_2|>1$ for all $k\ge 6$.

Finally, we must show that $r_2$ is nonreal. When $k$ is odd, $r_1$ is the only real root of $x^k-x^{k-1}-1$, so clearly $r_2$ is nonreal. When $k$ is even, $x^k-x^{k-1}-1$ has two real roots: the positive root $r_1$ and a negative root. However, the negative root lies between $-1$ and 0 and is thus not $r_2$ for $k\ge 6$, since $|r_2|>1$. \end{proof}

\begin{lemma} \label{lem:reccoeffs} Let $k\ge 6$. If $a_0,a_1,a_2,\ldots$ is a sequence of positive integers satisfying $a_n=a_{n-1}+a_{n-k}$ for all sufficiently large $n$, then, with notation as in Lemma~\ref{lem:linrecgenterm}, $r_1$ is real, $\beta_1>1$, $|r_1|>|r_2|=|r_3|>|r_4|$, and $\beta_2,\beta_3\neq 0$. Furthermore, $\beta_3=\bar{\beta}_2$, where the bar denotes complex conjugation. \end{lemma}

\begin{proof} By Lemma~\ref{lem:posdom} and the assumption that $a_n$ is positive and satisfies the recurrence $a_n=a_{n-1}+a_{n-k}$ for all sufficiently large $n$, it follows immediately $r_1$ is real, $\beta_1>1$, and $|r_1|>|r_2|$. Furthermore, $r_2$ is nonreal by Lemma~\ref{lem:r2stuff}. Since nonreal roots of polynomials with real coefficients come in complex conjugate pairs, it follows that the complex conjugate $\bar{r}_2$ of $r_2$ is also a root of $x^k-x^{k-1}-1$. Thus $|r_2|=|r_3|$. By Lemma~\ref{lem:2permag}, $|r_2|>|r_4|$.

To see that $\beta_2,\beta_3\neq 0$, note that all the roots of $\chi(x)$, except possibly the two sixth roots of unity satisfying $x^2-x+1$, have the same minimal polynomial over $\QQ$ by Lemma~\ref{lem:irreducible}. Since $|r_2|>1$, $r_2$ is not one of those roots of unity. Thus $r_1,r_2,r_3$ are all roots of the same irreducible factor of $\chi$, and since $\beta_1\neq 0$, Lemma~\ref{lem:galois} implies that $\beta_2,\beta_3\neq 0$ as well.

To see that $\beta_3=\bar{\beta_2}$, note that since $\Gal(\CC(x)/\RR(x))=\{1,z\mapsto\bar{z}\}$ acts on the $\frac{\beta_i}{1-r_ix}$'s in the partial fraction decomposition of $\sum_{n=0}^\infty a_nx^n$ and complex conjugation sends $r_2$ to $r_3$, it must send $\frac{\beta_2}{1-r_2x}$ to $\frac{\beta_3}{1-r_3x}$. Thus $\bar{\beta}_2=\beta_3$. \end{proof}

\section{Stability}

We now come to the main result of the paper:

\begin{thm} \label{thm:stability} For any $\alpha \geq 1$, there exists a half-open interval $I_{\alpha} = [ \alpha_0, \alpha_1 )$ containing $\alpha$ such that for any $\beta\in I_{\alpha}$, the sequence $P^\beta_i$ is the same as the sequence $P^\alpha_i$, and for all $\beta\not\in I_\alpha$, the two sequences are not the same, in that there is some integer $i$ for which $P^\alpha_i\neq P^\beta_i$. \end{thm}

Before we prove Theorem~\ref{thm:stability}, let us take a look at why it ought to be true, by means of a typical example. Let us suppose that $\alpha=3$ and look at the sequence $P^3_n$. This sequence begins \[P^3_n:0,1,2,3,4,6,8,11,15,21,29,40,55,\ldots,\] with $P^3_n=P^3_{n-1}+P^3_{n-4}$ for sufficiently large $n$. For example, $P^3_8=15$. To compute $P^3_9$, we must add to $P^3_8=15$ the unique $P^3_i$ for which \begin{equation} \label{eq:3cutoff} 3P^3_{i-1}< P^3_8\le 3P^3_i,\end{equation} which is 6. Thus $P^3_9=15+6=21$. If we were to increase 3 to $\frac{15}{4}$ and all the previous $\mcP$-positions in $\frac{15}{4}$-\textsc{tag} agreed with those of 3-\textsc{tag}, then the left inequality in~(\ref{eq:3cutoff}) with 3 replaced with $\frac{15}{4}$ would fail. Thus if all the $\mcP$-positions of 3-\textsc{tag} and $\frac{15}{4}$-\textsc{tag} agree up to 15, then the next term is definitely different.

We can perform analogous calculations starting from any term of the sequence $P^3_n$. If $\alpha>3$, the only way that the sequence $P^\alpha_n$ could differ from $P^3_n$ is if $\alpha$ is greater than the analogous ratio, starting with some term of $P^3_n$. In fact, one of these ratios is $\frac{21}{6}=\frac{7}{2}$, so the $\mcP$-positions of $\alpha$-\textsc{tag} are only equal to those of 3-\textsc{tag} when $3\le\alpha<\frac{7}{2}$.

The proof of Theorem~\ref{thm:stability} is now reduced to showing that, for any $\alpha$, the infimum of the sequence of such ratios is achieved. In particular, since all the ratios are greater than $\alpha$, it follows that this infimum is strictly greater than $\alpha$.

To this end, we introduce some notation for these ratios. Fix an $\alpha$, and define a sequence $Q_k=Q_k^\alpha$ by setting
%
%
\[ Q_k^\alpha = \frac{\widehat{P}_k^\alpha}{P_k^\alpha}, \]
where \[\widehat{P}_k^\alpha=\min\{P_i^\alpha\in T(\alpha):P_i^\alpha>\max(W_\alpha(P_k^\alpha))\}\] is the smallest term in the sequence $P_i^\alpha$ greater than all the elements of the window of $P_k^\alpha$. Alternatively, $\widehat{P}_k^\alpha=P_{j+1}^\alpha$, where $P_j^\alpha=\max(W_\alpha(P_k))$. As discussed above, the next $\beta>\alpha$ for which there exists an $i$ such that $P^\beta_i\neq P^\alpha_i$ is $\inf_k\{Q_k^\alpha\}$.

The following lemma will be key to proving Theorem~\ref{thm:stability}.

\begin{lemma} \label{lem:osc} Let $\alpha\ge 2$ be a real number. The sequence $Q_n=Q_n^\alpha$ converges to some real number $r_1>1$, and $Q$ oscillates around the point of its convergence, in the sense that there are arbitrarily large integers $n$ such that $Q_n>r_1$, as well as arbitrarily large integers $n$ such that $Q_n<r_1$. \end{lemma}

\begin{proof}
There is some positive integer $k\ge 2$ such that the sequence $P^{\alpha}$ satisfies the linear recurrence of $P_n = P_{n - 1} + P_{n - k}$ for all sufficiently large $k$. When $k\le 5$, the remainder of the proof requires minor modifications since we cannot quite use Lemma~\ref{lem:reccoeffs}, but most of it still works in that case as well. The cases $k\le 5$ can also be checked by hand if desired. When $k=2$, $r_2$ is real, so a slightly different argument must be made, but again, most of the proof still works. From now on, we will assume that $k\ge 6$.

Since we are interested in the limiting or tail behavior of the sequence $Q_n$, we may ignore the initial terms, where $P_n$ does not satisfy the eventual recurrence $P_n=P_{n-1}+P_{n-k}$. Let us consider the characteristic polynomial of the recurrence 
\[ \chi(x)=x^k - x^{k - 1} - 1 = 0, \]
and let the roots of this polynomial be $r_1, r_2, r_3, \ldots , r_k$, where $|r_1|\ge|r_2|\ge\cdots\ge|r_k|$. By Lemma~\ref{lem:linrecgenterm}, we know that there exist $\beta_1,\ldots,\beta_k\in\mathbb{C}$ such that
\[ P_n = \beta_1r_1^n + \beta_2r_2^n + \beta_3r_3^n + \cdots + \beta_kr_k^n \] for all sufficiently large values of $n$. The sequence of ratios eventually converges to $r_1^k$.
We want to know if the sequence of ratios oscillate below and above $r_1^k$. Thus, we study the following sequence 
\[ \frac{P_{n + k}}{P_n} - r_1^k \]
We have
\begin{align*}
\frac{P_{n + k}}{P_n} - r_1^k &= \frac{r_1^{n + k} + \frac{\beta_2}{\beta_1}r_2^{n + k} + \frac{\beta_3}{\beta_1} r_3^{n + k} + \cdots + \frac{\beta_k}{\beta_1} r_k^{n + k}}{r_1^{n} + \frac{\beta_2}{\beta_1}r_2^{n} + \frac{\beta_3}{\beta_1} r_3^{n} + \cdots + \frac{\beta_k}{\beta_1} r_k^{n}} - r_1^k \\
&= \frac{\beta_2r_2^n (r_2^k - r_1^k) + \beta_3r_3^n (r_3^k -r_1^k) + \cdots + \beta_kr_k^n (r_k^k - r_1^k)}{\beta_1r_1^n + \beta_2r_2^n + \beta_3r_3^n + \cdots + \beta_kr_k^n}.
\end{align*}
The denominator is positive since it is just equal to $P_n$. We must show that the numerator is positive for infinitely many $n$ and negative for infinitely many $n$. Note that $\beta_i, (r_i^k - r_1^k)$ are all constants; only $r_2^n, r_3^n, \ldots, r_k^n$ change as a function of $n$. 

At this point, we are trying to determine if 
\begin{equation} \beta_2r_2^n (r_2^k - r_1^k) + \beta_3r_3^n (r_3^k -r_1^k) + \cdots + \beta_kr_k^n(r_k^k - r_1^k) \label{eq:oscillatory} \end{equation}
displays oscillatory behavior as a function of $n$. By Lemma~\ref{lem:reccoeffs}, $\beta_2,\beta_3\neq 0$ and $|r_3|>|r_4|$, so for sufficiently large values of $n$, $r_2^n$ and $r_3^n$ will dominate the rest of the terms, so for sufficiently large values of $n$, the sign of~(\ref{eq:oscillatory}) will be the same as the sign of $\beta_2r_2^n(r_2^k-r_1^k)+\beta_3r_3^n(r_3^kr_1^k)$. Note that the behavior of terms $r_2$ and $r_3$, the next roots of largest magnitude, are what really determine the behavior of the entire sequence for sufficiently large $n$. Let us write \[r_2^k-r_1^k=\rho e^{i\phi},\qquad r_3^k-r_1^k=\rho e^{-i\phi}\] and \[r_2=re^{i\theta},\qquad r_3=re^{-i\theta}.\] 
Then we have \[r_2^n=r^ne^{in\theta},\qquad r_3^n=r^ne^{-in\theta}.\]
Thus we have 
\[ \beta_2r_2^n(r_2^k-r_1^k)+\beta_3r_3^n(r_3^k-r_1^k) = \beta_2r^ne^{in\theta}\rho e^{i\phi} + \beta_3r^ne^{-in\theta}\rho e^{-i\phi}. \] Since $\bar{\beta}_2=\beta_3$ by Lemma~\ref{lem:reccoeffs}, we may also write \[\beta_2=se^{i\psi},\qquad \beta_3=se^{-i\psi},\] so that we have
\[\beta_2r_2^n(r_2^k-r_1^k)+\beta_3r_3^n(r_3^k-r_1^k) = 2sr^n\rho \cos(\psi+n\theta+\phi).\]
Since $\phi$, $\psi$, and $\theta$ are fixed and $\theta\not\equiv 0\pmod{2\pi}$, we know that $\cos(\psi+n\theta+\phi)$ is positive for infinitely many values of $n$ and negative for infinitely many values of $n$. Thus there are infinitely many values of $n$ for which $Q_n>r_1$, and infinitely many values of $n$ for which $Q_n<r_1$, as desired.
\end{proof}

Using Lemma~\ref{lem:osc}, we can now prove Theorem~\ref{thm:stability}.

\begin{proof}[Proof of Theorem~\ref{thm:stability}]
Define $Q_i^{\alpha}$ by 
 \[Q^{\alpha}_i = \frac{P_{j + 1}^\alpha}{P_i^\alpha} \] 
where $P_j = \max(W_\alpha(P_i))$. We established in Lemma~\ref{lem:osc} that $Q_\alpha$ has a minimum. Say we have some $\alpha < \beta < \min(Q_\alpha)$. We will show that $P^{\beta} = P^{\alpha}$. Say $P^{\beta} \not = P^{\alpha}$. A sequence of $T(\alpha)$ positions is determined by 
\[ P_{i + 1} = P_i + P_j \quad \text{if} \quad P_j \in W_\alpha(P_i)\]
If $P^{\beta} \not = P^{\alpha}$, this implies there is a first occurrence of $i$ such that $W_\beta(P^{\beta}_i) \not = W_\alpha(P^{\alpha}_i)$. Since $\beta > \alpha$, this means that $\max(W_\beta(P^{\beta}_i)) > \max(W_\alpha(P^{\alpha}_i))$. Say $P_j =\max(W_\alpha(P^{\alpha}_i))$. Then $\max(W_\beta(P^{\beta}_i)) \geq P_{j + 1}$, which means \[P_{j + 1} \leq \beta \cdot P_i,\]
or \[Q_i=\frac{P_{j + 1}}{P_i} \leq \beta,\]
contrary to our assumption. Next, we show that if $P^{\beta} = P^{\alpha}$, then $\beta < \min Q_\alpha$. Say $\beta \geq \min Q_\alpha$. Let the index at which $Q_\alpha$ reaches its minimum be $k$. The sequence $T(\alpha)$ is determined by 
\[ P_{i + 1} = P_i + P_j \quad \text{if} \quad P_j \in W_\alpha(P_i).\]
We will show that $\max(W_\beta(P^{\beta}_k)) > \max(W_\alpha(P^{\alpha}_k))$. Let $\max(W_\alpha(P^{\alpha}_k)) = P_x$. Thus, $\min Q_\alpha = \frac{P_{x+1}}{P_k}$. Note that 
\[\alpha P_{k - 1} < P_{x} \leq \alpha \cdot P_{k} \]
and
\[P_{x} < P_{x + 1}\leq \beta \cdot P_{k}.\]
Therefore, $\max(W_\beta(P^{\beta}_k)) \geq P_{x + 1} > P_x = \max(W_\alpha(P^{\alpha}_k))$. But since $W_\beta(P^{\beta}_k) \not = W_\alpha(P^{\alpha}_k)$, $P^{\beta} \not = P^{\alpha}$, which is a contradiction. 
\end{proof}

In short, the $T(\alpha)$ positions remain the same in certain intervals as $\alpha$ changes. Table~\ref{tab:stabints} shows the first several stable intervals. Note that the same eventual recurrence can describe more than one set of $T(\alpha)$ positions, as seen with the recurrence $P_n = P_{n-1} + P_{n -5}$. This is because it takes longer for the recurrence to start holding when $\frac{7}{2}\le\alpha<\frac{11}{3}$ than it does when $\frac{11}{3}\le\alpha<\frac{43}{11}$. 

\begin{table}\begin{tabular}{|c|c|c|} \hline Range & Eventual recurrence & Initial conditions \\ \hline $1\le\alpha<2$ & $P_n=P_{n-1}+P_{n-1}$ & 0,1 \\ \hline $2\le\alpha<\frac{5}{2}$ & $P_n=P_{n-1}+P_{n-2}$ & 0,1,2 \\ \hline $\frac{5}{2}\le\alpha<3$ & $P_n=P_{n-1}+P_{n-3}$ & 0,1,2,3,5 \\ \hline $3\le\alpha<\frac{7}{2}$ & $P_n=P_{n-1}+P_{n-4}$ & 0,1,2,3,4,6 \\ \hline $\frac{7}{2}\le\alpha<\frac{11}{3}$ & $P_n=P_{n-1}+P_{n-5}$ & 0,1,2,3,4,6,8,11,15,21 \\ $\frac{11}{3}\le\alpha<\frac{43}{11}$ & $P_n=P_{n-1}+P_{n-5}$ & 0,1,2,3,4,6,8,11 \\ \hline $\frac{43}{11}\le\alpha<4$ & $P_n=P_{n-1}+P_{n-6}$ & 0,1,2,3,4,6,8,11,14,18,24,32,43 \\ $4\le\alpha<\frac{13}{3}$ & $P_n=P_{n-1}+P_{n-6}$ & 0,1,2,3,4,5,7,9,12 \\ \hline $\frac{13}{3}\le\alpha<\frac{31}{7}$ & $P_n=P_{n-1}+P_{n-7}$ & 0,1,2,3,4,5,7,9,12,15,19,24,31,40,52 \\ $\frac{31}{7}\le\alpha<\frac{9}{2}$ & $P_n=P_{n-1}+P_{n-7}$ & 0,1,2,3,4,5,7,9,12,15,19,24,31 \\ $\frac{9}{2}\le\alpha<\frac{14}{3}$ & $P_n=P_{n-1}+P_{n-7}$ & 0,1,2,3,4,5,7,9,11,14,18 \\ \hline  \end{tabular}
\caption{Stable intervals for \textsc{$\alpha$-tag}}
\label{tab:stabints}
\end{table}

\section{Cutoffs}

\begin{defn}
A \textit{cutoff} is some number $\alpha\ge 1$ such that, for any $\beta < \alpha$, the sequences $P^\alpha_n$ and $P^\beta_n$ are not identical.
\end{defn}

In other words, the cutoffs are the endpoints of the stable intervals of Theorem~\ref{thm:stability}. The first few cutoffs are $1, 2, \frac{5}{2}, 3, \frac{7}{2}, \frac{11}{3}, \frac{43}{11}, 4, \frac{13}{3}$. 

\begin{cor} All cutoffs are rational numbers. \end{cor}

\begin{proof} The cutoffs are infima of sequences of rational numbers, and these infima are always achieved and hence rational. \end{proof}

In order to investigate the cutoffs more thoroughly, we consider a new sequence generated from the sequence $P^\alpha_i$. 

\begin{defn} The \textit{sequence of indices of recurrence} $S^{\alpha}_i$ is defined by \[S^\alpha_i=\max\{j:P^\alpha_i+P^\alpha_{i+j-1}=P^\alpha_{i+j}\}.\] \end{defn}

\begin{ex} Let $\alpha=\frac{7}{2}$. Then we have the following initial values of $P_i$ and $S_i$: \begin{center}\begin{tabular}{c|cccccccccccccc} $P_i$ & 1 & 2 & 3 & 4 & 6 & 8 & 11 & 15 & 21 & 27 & 35 & 46 & 61 \\ \hline $S_i$ & 3 & 4 & 4 & 4 & 5 & 5 & 5 & 5 & 5 & 5 & 5 & 5 & 5 \end{tabular} \end{center} \end{ex}

The next lemmas are from Schwenk's paper. 

\begin{lemma}[\cite{Schwenk70}] \label{lem:schwenk1}
For some \textsc{$\alpha$-tag}, with $T(\alpha)$ positions $P_t$, if $\alpha \cdot P_{i - 1} < P_j \leq \alpha \cdot P_i$, then $\alpha \cdot P_{i + 1} \geq P_{j + 1}$. 
\end{lemma}

\begin{lemma}[\cite{Schwenk70}] \label{lem:dec}
For every $i$, we have $S^\alpha_i\le S^\alpha_{i+1}$.
\end{lemma}
\begin{proof}
Recall that the window $W_\alpha(P^{\alpha}_i)$ of $P_i^\alpha \in T(\alpha)$ is 
\[W_\alpha(P^{\alpha}_i) = \{ P_j^\alpha \in T(\alpha) : P_i^\alpha + P_j^\alpha=P_{j+1}^\alpha \in T(\alpha)\}. \]
We proved previously that  
\[W_\alpha(P^{\alpha}_i) = \{ P_j \in T(\alpha) : \alpha \cdot P_{i - 1} < P_j \leq \alpha \cdot P_{i} \}. \]
Say $P_j = \max \{ W_\alpha(P^{\alpha}_i) \}$. Since 
\[ \alpha \cdot P_{i - 1} < P_j \leq \alpha \cdot P_{i}, \]
Lemma ~\ref{lem:schwenk1} implies that $P_{j + 1} \leq \alpha \cdot P_{i + 1}$. Next, we prove that $\alpha \cdot P_i < P_{j + 1}$. We prove this with contradiction. Assume $P_{j + 1} \leq \alpha \cdot P_i$. This would imply
\[  \alpha \cdot P_{i - 1} < P_j < P_{j + 1} \leq \alpha \cdot P_{i}. \]
This means $P_{j + 1} \in W_\alpha(P^{\alpha}_i)$. However, we said $P_j = \max \{ W_\alpha(P^{\alpha}_i) \}$ so this is a contradiction. Therefore, we have shown that 
\[  \alpha \cdot P_{i} < P_{j + 1} \leq \alpha \cdot P_{i + 1}. \]
So, from assumption, $P_j = \max \{ W_\alpha(P^{\alpha}_i) \}$, and Lemma~\ref{lem:schwenk1} implies $P_{j + 1} \in W_\alpha(P^{\alpha}_{i+1})$. Thus $S^{\alpha}_i=j - i - 1$ and $S^{\alpha}_{i + 1}\ge j + 1 - (i + 1) - 1 = j - i - 1$. Therefore, Lemma~\ref{lem:schwenk1} implies that $S^{\alpha}$ is a monotonically increasing sequence. 
\end{proof}

\begin{lemma}[\cite{Schwenk70}] \label{lem:schwenk2}
Suppose there exists a $j$ such that \begin{equation} P_{j + i + 1} = P_{j + i} + P_{j + i - k}\label{eq:enoughis}\end{equation} for all $i\in\{0,1,\ldots,k+1\}$. Then (\ref{eq:enoughis}) holds for every nonnegative integer $i$. 
\end{lemma}

\begin{lemma} The number of cutoffs in any closed interval $[a,b]$ is finite. \end{lemma}

\begin{proof}
We first prove that the number of eventual recurrences in the interval is finite. There are at least two ways of doing this. One way would be to prove that the degree $k$ of the eventual recurrence increases with $\alpha$; this is true, but we have not proven it. An alternative approach is to use a result of~\cite{Zieve96}. Zieve proves that
\[\frac{\log(\alpha-1)}{\log(\alpha)-\log(\alpha-1)}\le k\le\frac{\log(\alpha)}{\log(\alpha+1)-\log(\alpha)}.\] It follows that for all $\alpha\in[a,b]$, we have
\[ \frac{\log(a-1)}{\log{a} - \log{(a-1)}} \leq k \leq \frac{\log(b)}{\log{(b+1)} - \log{b}}. \] 
Since $k$ is an integer, there are only finitely many eventual recurrences in a closed interval. Thus it remains to show that there are only finitely many sequences with the eventual recurrence $P_n = P_{n-1} + P_{n - k}$. 
From Lemma ~\ref{lem:dec}, we know that $S^{\alpha}$ is an increasing sequence. By Lemma~\ref{lem:schwenk2}, any positive integer $m \leq k$ can appear at most $m + 1$ times in the sequence $S^\alpha$. Thus there are only finitely many possible initial strings of the sequence $S^\alpha$ before the sequence stabilizes at $k$. 
It follows that there are only finitely many cutoffs in any closed interval $[a,b]$.
\end{proof}

\begin{thm} \label{intcutoff} Every integer $n\ge 2$ is a cutoff. \end{thm}

\begin{proof}
Let the last cutoff before $n \in \mathbb{Z}$ where $n \geq 2$ be $n_0$ and let $\eps = \frac{n - n_0}{2}$. We consider the \textsc{$\alpha$-tag} where $\alpha = n_0 + \eps$. We generate the sequence of possible $\alpha$'s. Take the window for the first term, 1. The last term of the window for 1 is $\left \lfloor{\alpha}\right \rfloor$. Thus, the next term is $n$. We know that the sequence of possible alphas does not have any number smaller than $n$ since we assumed $n_0$ to be the last cutoff before $n$. Thus, $n$ is the next cutoff. 
\end{proof}

We can also prove a generalization of this theorem.

\begin{thm} \label{thm:1/n-cutoff} Let $x \equiv 0 \pmod{n!}$ and $x > 0$. Then $x + \frac{1}{n}$ is a cutoff. \end{thm}

Before we prove Theorem~\ref{thm:1/n-cutoff}, let us explain the intuitive reason behind it, which we make precise using windows. Let $\alpha$ be the largest cutoff before $x+\frac{1}{n}$. Since all integers are cutoffs, we have $\alpha\ge x$. Thus the sequence $T(\alpha)$ starts with an arithmetic progression that increases by 1 each time, then increases by 2 each time, then by 3, and so forth. In fact, it begins with $0,1,2,\ldots,x,x+1,x+3,x+5,\ldots,2x+1,2x+4,2x+7,\ldots,3x+1,3x+5,\ldots,4x+1,\ldots,5x+1,\ldots,nx+1,nx+n+2$. Thus one of the $Q$'s that gives an upper bound for the next cutoff after $\alpha$ is $\frac{nx+1}{n}$.

\begin{proof}[Proof of Theorem~\ref{thm:1/n-cutoff}]
We proceed by induction on $n$, proving the given statement together with an auxiliary result that aids in the inductive step. The auxiliary result is that if $\alpha$ is the largest cutoff less than $x+\frac{1}{n}$, then $\max(W_\alpha(n))=nx-n+1$. For the original statement, the base case, $n=1$, is simply Theorem~\ref{intcutoff}. For the auxiliary statement, the largest cutoff less than $x+1$ is simply $x$ because the sequence $(\alpha)$ begins $0,1,2,\ldots,x+1$. The next term is $x+3$. Thus $\max(W_\alpha(1))=x$, as claimed. 

Now suppose that the result is true for $n$, and we'll prove it for $n+1$. Let $x\equiv 0\pmod{(n+1)!}$, and let $\alpha$ be the last cutoff before $x+\frac{1}{n+1}$. We consider the sequence $T(\alpha)$. Since $x\equiv 0\pmod{(n+1)!}$, we also have $x\equiv 0\pmod{n!}$, so $\max(W_\alpha(n))=nx-n+1$. Since $n+1\in T(\alpha)$, the next term in $T(\alpha)$ after $nx-n+1$ is in $W_\alpha(n+1)$, and that next term is $\max(W_\alpha(n))+n=nx+1$. Let us now compute $W_\alpha(n+1)$. It begins with $nx+1$, and it is an arithmetic progression with common difference $n+1$, so its elements are of the form $nx+1+k(n+1)$, where $nx+1+k(n+1)\in W_\alpha(n+1)$ if and only if $nx+1+k(n+1)\le \alpha(n+1)$. Since $x\le\alpha<x+\frac{1}{n+1}$, we have $nx+1+k(n+1)\in W_\alpha(n+1)$ if and only if $k<\frac{x}{n+1}$, so \[\max(W_\alpha(n+1))=(n+1)x+1+\left(\frac{x}{n+1}-1\right)(n+1)=(n+1)x-(n+1)+1,\] completing the induction.
\end{proof}

Theorem~\ref{thm:1/n-cutoff} show that for all integers $d$, there exists a cutoff whose denominator in lowest terms is $d$. In fact, it is quite common for rational numbers with small denominators to appear as cutoffs, even when they are not guaranteed by Theorem~\ref{thm:1/n-cutoff}. For instance, all half-integers from $\frac{5}{2}$ to $\frac{29}{2}$ are cutoffs, but $\frac{31}{2}$ is not. The next few half-integers that are not cutoffs are $\frac{43}{2}$, $\frac{75}{2}$, $\frac{79}{2}$, and $\frac{95}{2}$. It would be interesting to investigate the nature of the cutoffs with a given denominator. For example, for those arithmetic progressions of rational numbers such that Theorem~\ref{thm:1/n-cutoff} does not guarantee that all are cutoffs, is it true that infinitely many are cutoffs and infinitely many are not cutoffs? Or are there other arithmetic progressions containing only cutoffs or only noncutoffs (or all but finitely many cutoffs or noncutoffs)?

We have written a number of computer programs to aid the calculations of the sequences $T(\alpha)$ and the generation of cutoffs\footnote{Computer programs as well as cutoff data can be found at \url{https://github.com/sherrysarkar}}. Based on the data displayed in Table~\ref{tab:numcutoffs} and Figure~\ref{fig:cutoffs}, we make the following conjecture:

\begin{conjecture}
Let $\gamma(n)$ be the number of cutoffs up to $n$. Then $\lim_{n\to\infty} \frac{\gamma(n)}{n^2}$ exists and is nonzero.
\end{conjecture}

\begin{table}
 \begin{tabular}{|c|c|c|c|c|c|c|c|c|c|c|c|c|c|} 
 \hline
 $n$ & 2.5 & 3 & 3.5 & 4 & 4.5 & 5 & 5.5 & 6 & 10 & 20 & 30 & 40 & 75 \\ 
 \hline
 $\gamma(n)$ & 3 & 4 & 5 & 8 & 11 & 14 & 18 & 21 & 74 & 424 & 1144 & 2100 & 9084 \\
 \hline
\end{tabular}
\caption{Number of Cutoffs from 1 to $n$}
\label{tab:numcutoffs}
\end{table}

\begin{figure}
\includegraphics[width=8cm]{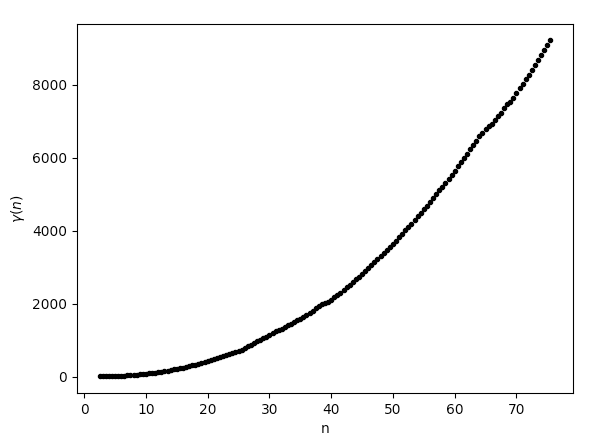}
\caption{$\gamma(n)$ versus $n$}
\label{fig:cutoffs}
\end{figure}

\bibliographystyle{alpha}
\bibliography{alphagame}

\end{document}